\documentclass[12pt]{iopart}

\newcommand{\R}{\mathbb{R}}

\newcommand{\E}{\mathbb{E}}

\usepackage{amssymb}
\usepackage{graphicx}
\usepackage{pstricks}
\usepackage{iopams}
\usepackage{amsthm}
\usepackage{color}
\usepackage{hyperref}
\newtheorem{thm}{Theorem}[section]
\newtheorem{prp}{Proposition}[section]
\newtheorem{lem}{Lemma}[section]

\newtheorem{definition}{Definition}[section]

\begin{document}
\title[Stopping criterion for constrained maximum likelihood algorithms]
{Regularization of constrained maximum likelihood iterative algorithms by means of statistical stopping rule}
\author{Federico Benvenuto and Michele Piana}
\address{DIMA, Universit\`a di Genova, Via Dodecaneso 35, 16100, Italy}
\ead{benvenuto@dima.unige.it}

\begin{abstract}
In this paper we propose a new statistical stopping rule for constrained maximum likelihood iterative algorithms applied to ill-posed inverse problems. To this aim we extend the definition of Tikhonov regularization in a statistical framework and prove that the application of the proposed stopping rule
to the Iterative Space Reconstruction Algorithm (ISRA) in the Gaussian case and Expectation Maximization (EM) in the Poisson case leads to well defined regularization methods according to the given definition. We also prove that, if an inverse problem is genuinely ill-posed in the sense of Tikhonov, the same definition is not satisfied when ISRA and EM are optimized by classical stopping rule like Morozov's discrepancy principle, Pearson's test and Poisson discrepancy principle. The stopping rule is illustrated in the case of image reconstruction from data recorded by the Reuven Ramaty High Energy Solar Spectroscopic Imager (RHESSI). First, by using a simulated image consisting of structures analogous to those of a real solar flare we validate the fidelity and accuracy with which the proposed stopping rule recovers the input image. Second, the robustness of the method is compared with the other classical stopping rules and its advantages are shown in the case of real data recorded by RHESSI during two different flaring events.
\end{abstract}

\maketitle

\section{Introduction}

Maximum likelihood (ML) is a standard approach to parameter estimation in statistics, and provides estimates of the solution even for statistical inverse problem. Once the statistical distribution of data is known, its application is straightforward and is widely used in many different fields \cite{Tarantola:2004:IPT:1062404}, from astronomical image reconstruction to medical imaging where the data distribution is often Poisson; from geophysical to elastic waveforms inverse problem, where the data are mainly modeled by Gaussian noise.

ML estimates can be effectively constrained to be non negative. In the case of Poisson noise, the constrained ML approach leads to the Expectation Maximization (EM) algorithm \cite{EM}. In the case of white Gaussian noise, the constrained ML approach leads to the Iterative Space Reconstruction Algorithm (ISRA) \cite{ISRA}. Although both these algorithms have the convergence property, the properties of the limit solution have not yet been fully investigated. What is known is that such solution, in the case of experimental data, although non negative, is not acceptable from a physical viewpoint, since the intrinsic ill-posedness of the inverse problem induces noise amplification. There are two ways to regularize a statistical inverse problem: first, with an add of information on the solution realized by a prior probability in a Bayesian framework, or, second, without adding information on the solution but simply stopping the iterative approximation process before getting the limit solution \cite{MR726773}. 
This paper focuses on this second approach.

Among the many different stopping criteria available in literature, Morozov's discrepancy principle \cite{Morozov:105020}, Pearson's cumulative test \cite{Veklerov, herbert} and Generalized Cross validation \cite{Coakley,GCV} can be applied to rather general schemes, while the updating coefficients method \cite{Tzanakos} and the more recent Poisson discrepancy criterion \cite{bertero2010} have been specifically designed for EM.
Some of these criteria may not work when the inverse problem is genuinely ill-posed in a Tikhonov sense \cite{Tikhonov}, i.e. when the data do not belong to the range of the forward operator. The object of this paper is to introduce an optimal stopping rule for constrained iterative ML methods which is characterized by two properties. First, it is motivated by statistical arguments; second it works under completely general conditions (i.e. including the case of genuinely ill-posed problems).
We apply this criterion to EM and ISRA and we show its effectiveness in the case of both synthetic (but realistically simulated) and real data. Finally we prove that this stopping rule makes these two iterative algorithms well defined regularization methods, according to a regularization definition which extends the classical Tikhonov definition to a statistical framework.

In \Sref{sec:Constrained} we describe the constrained ML approach to a statistical inverse problem, showing how to derive a general iterative algorithm given the likelihood, and we complete this overview applying the method to the Gaussian and Poisson cases. In \Sref{sec:Regularization} we introduce a definition of regularization for constrained ML problems extending the Tikhonov one to a statistical framework and we also show that classical stopping rules satisfy this definition only when the problem is not genuinely ill-posed. In \Sref{sec:Criterion} we propose the new statistical criterion and we prove that it is a well-defined regularization method both in the Gaussian and the Poisson case. In \Sref{sec:Experiments} we perform some numerical applications with Gaussian and Poisson noise respectively, and we also show an application to a real case reconstructing solar X-ray images starting from count data collected by an on-orbit satellite.

\section{Constrained ML problem}
\label{sec:Constrained}

We denote by $x = \{x_j \}_{j=1,\ldots,M}$ the unknown parameters, where $j$ is, in general, a multi-index. Moreover, we denote by $y = \{y_i\}_{i=1,\ldots,N}$ the detected signal, where $i$ can be in general again a multi-index. $N$ and $M$ are, respectively, the number of data
and unknown parameters of the reconstruction problem. Finally, let the relation between unknown parameters and the data be described by a linear system.
We denote by $H$ the matrix describing the transformation from the parameter space to the signal space. We suppose each element of $H$ is positive, i.e.
\begin{equation}
H_{ij} > 0 ~~~ ~~~ \forall ~ i=1,\ldots,N, ~ j=1,\ldots,M ~.
\end{equation}
Then, the forward model can be written in the form
\begin{equation}
y=Hx ~~~ .
\label{model}
\end{equation}

A standard statistical approach for estimating the parameter $x$ given $y$, is the Maximum Likelihood (ML) method. It is based on the assumption that the data vector $y$ is an observed value of random vector $Y$ with mean $Hx$. In other words, denoting by $\eta$ a generic random process, we can write
\begin{equation}
Y_i \sim \eta_i (Hx) ~~~ .
\end{equation}
The density function of the vector $Y$ is given by the joint probability density function $p_\eta(y,Hx)$. When this density is thought of as a function of $x$ given $y$ we call it the likelihood and we write
\begin{equation}
\mathcal L_y(x) = p_\eta(y,Hx) ~~~ .
\end{equation}

Once the matrix $H$ and the data $y$ are given, ML obtains the solution of $x$ when the likelihood reaches its maximum value. However, in the majority of cases physical motivations regarding the nature of the problem imply that the components of parameter $x$ has to be non-negative. Consequently, the ML estimator is constrained as
\begin{equation}
\hat x = \arg \max_{x\in \mathcal C} ~ \mathcal{L}_y(x) ~~~ ,
\label{ML_problem}
\end{equation}
where $\mathcal C=\{x ~|~ x_j \geq 0 ~,~\forall~ j=1,\ldots,M\}$
is the non-negative orthant.

Usually, it is more convenient to minimize the negative logarithm of the likelihood instead of maximizing it. The constrained ML problem is therefore equivalent to 
\begin{equation}
\hat x = \arg \min_{x\in \mathcal C} ~ L_y(x) ~~~,
\label{neglog_problem}
\end{equation}
where $L_y(x)=-\log(\mathcal L_y(x))$. Indeed, when both the negative logarithm of the likelihood function and the constraint are convex, the necessary and sufficient conditions for $x$ to be the constrained ML estimator are the Karush-Khun-Tucker (KKT) conditions \cite{Boyd}, which in this particular case take the form
\begin{equation}
x ~ \nabla L_y(x) = 0 ~~~ , ~~~ \mathrm{with} ~~~ x \geq 0
\label{KKT}
\end{equation}
where multiplication and inequality between vectors are done element-by-element.

Splitting the gradient into the positive part $V(x)$ and the negative part $U(x)$ transforms \eref{KKT} into a fixed point equation \cite{lanteri2002}; then applying the successive approximation method leads to the multiplicative iterative algorithm
\begin{equation}
x^{(k+1)} = x^{(k)} \frac{U(x^{(k)})}{V(x^{(k)})} ~~~ .
\label{iterative_approx}
\end{equation}
Even if there is no general proof of convergence for these algorithms, in the particular convex cases we will discuss this technique leads to well-known algorithms for which proof of convergence has been done. It has been conjectured \cite{night_sky_conjecture} that the minimizer of the constrained ML problem should be sparse in a pixel space, but the problem of giving a parametric form of this solution is still open. Moreover, in \cite{0266-5611-26-2-025004} authors have shown with a numerical experiment that this conjecture is verified when the data are perturbed by Gaussian noise.

When $\eta$ is a vector of independent and identically distributed Gaussian variables with mean given by $Hx$ and variance equal to $\sigma^2$, the constrained ML problem \eref{ML_problem} is equivalent to the minimization of the Least Squares function
\begin{equation}
D_{LS}(y,x) = \| Hx - y\|^2 ~~~ ,
\label{LS}
\end{equation}
under the non-negativity constraint \cite{bertero1998book}. Consequently, KKT conditions \eref{KKT} lead to the algorithm
\begin{equation}
x^{(k+1)} = x^{(k)} \frac{H^T y}{H^T Hx^{(k)}} ~~~ ,
\label{ISRA}
\end{equation}
which is known as Iterative Space Reconstruction Algorithm (ISRA), it has been introduced \cite{ISRA} as an acceleration of EM and it is convergent to the constrained minimum of $D_{LS}$ \cite{depierro}.

Analogously, when $\eta$ is a vector of independent and identically distributed Poisson variables with parameter given by $Hx$, the negative logarithm of the likelihood \eref{neglog_problem} is equivalent to the Kullback Leibler divergence
\begin{equation}
D_{KL}(y,x) = \sum_{i=1}^N  y_i \log \frac{y_i}{(Hx)_i} + (Hx)_i - y_i ~~~ .  
\label{KL}
\end{equation}
Hence, the constrained ML problem \eref{ML_problem} is equivalent to the minimization of $D_{KL}$ under the non-negativity constraint \cite{bertero1998book}. In the Poisson case the KKT conditions \eref{KKT} lead to the following iterative algorithm
\begin{equation}
x^{(k+1)} = \frac{ x^{(k)} }{H^T 1} H^T \frac{y}{Hx^{(k)}} ~~~ ,
\label{em_alg}
\end{equation}
which is known as Expectation Maximization \cite{EM}, or also as Richardson Lucy algorithm \cite{RICHARDSON:72,lucy} when $H$ represents a convolution operator and it is convergent to the constrained minimum of $D_{KL}$ \cite{Shepp}.

The limit solutions reached by EM and ISRA satisfy two analogous properties. In fact, let us first observe that, if
\begin{equation}
H(\mathcal C) := \{ ~ y' \in \R^N ~ | ~ y' = \sum_{j=1}^M a_j H_{ij} ~,~ a_j \geq 0 ~ \} ~~~ ,
\end{equation}
then
\begin{lem}
\label{ML_minimum}
The minimum of the functions $D_{LS}$ and $D_{KL}$ is zero if and only if $y \in H(\mathcal C)$.
\end{lem}

\begin{proof}
We start by remarking that $D_{LS}$ and $D_{KL}$ are nonnegative functions. So if they assume the zero value, this is the minimum. For both $D_{LS}$ and $D_{KL}$ it is evident that the zero value is assumed if and only if $Hx=y$. Therefore, the minimum of the functions $D_{LS}$ and $D_{KL}$
is zero if and only if there exists at least one point $\bar x \in \mathcal C$ such that $H\bar x = y$. Such a point exists if and only if the data $y$ belongs to $H(\mathcal C)$.
\end{proof}

Since EM and ISRA converge to a minimum of the functionals $D_{KL}$ and $D_{LS}$ respectively, Lemma \ref{ML_minimum} implies the following:
\begin{prp}
\label{zero_conv}
If $y \not\in H(\mathcal C)$ then the algorithms EM and ISRA converge to a solution $x^{(\infty)}$ such that
\begin{equation}
D_{KL} (y,x_{EM}^{(\infty)}) > 0 ~~~ \mathrm{and} ~~~
D_{LS} (y,x_{ISRA}^{(\infty)}) > 0
\end{equation}
respectively. 
\end{prp}

\section{Regularization for a constrained ML problem}
\label{sec:Regularization}

The noise corrupting the data $y$ makes the constrained ML solution not physically acceptable. In this case one can get an estimation of the unknown signal exploiting regularization. However, the definition of regularization given by Tikhonov \cite{Tikhonov} does not explicitly use statistical concepts, but takes place in the framework of functional analysis where the noise is modeled as a generic perturbation of the data in a metric space,
instead of as a random variable. In order to provide a statistical definition of regularization we utilize the concept of coefficient of variation \cite{DeGroot}.

Given a random variable $\mathcal N$ with mean $\mu$ and standard deviation $\sigma$, the coefficient of variation is defined as the ratio $\delta = \sigma / \mu$, which is an inversely proportional measure of the signal to noise ratio. Let us denote with $y_\delta=\mathcal N_\sigma (\mu)$
a realization of the random variable of a given coefficient of variation $\delta$, and let the mean $\mu$ be modeled by the action of a linear operator
on the unknown parameter $x$, i.e. $\mu=Hx$.

\begin{definition}
\label{stat_reg_def}
When $\mathcal L(x,y_\delta)$ is the likelihood associated with $y_\delta = \mathcal N_{\sigma} ( Hx )$ and $\mathcal C$ is a convex subset of $\R^M$,
an operator $R_\alpha: \Omega \subset \R^N \to \R^M$ is said to be a regularizing operator for the constrained ML problem
\begin{equation}
\label{con_max_lik}
\arg \max_{x\in \mathcal C} ~ \mathcal{L}(x,y_\delta)
\end{equation}
if the following conditions hold:
\begin{itemize}
\item [1)] $R_\alpha$ is defined on the range of $H$;
\item [2)] there exists a function $\alpha = \alpha(\delta)$ such that, when the coefficient of variation tends to zero, $R_{\alpha(\delta)}(y_\delta)$ tends to a constrained solution $\hat x$ of problem \eref{con_max_lik}, i.e.
\begin{equation}
\label{limit_cond}
\lim_{\delta \to 0} R_{\alpha(\delta)}(y_\delta) = \hat x.
\end{equation}
\end{itemize}
In the case of iterative methods, when an algorithm $x^{(k)}$ converges to a constrained ML solution $\hat x$, condition 2) can be restated as
\begin{itemize}
\item [2a)] there exists a function $k = k(\delta)$ such that,
\begin{equation}
\lim_{\delta \to 0} k(\delta) = \infty ~~~ .
\end{equation}
\end{itemize}
\end{definition}

We note that when $\mathcal N$ is a Gaussian variable, $H$ represents an invertible linear operator, and $\mathcal C = \R^N$, the constrained ML problem
becomes equivalent to the unconstrained least square problem, and so this definition coincides with the well-known Tikhonov definition (specifically, as $\delta$ tends to zero, the unconstrained ML solution $\hat x$ tends to the so-called generalized solution). When $\mathcal N$ is a Poisson variable the second statement of Definition \ref{stat_reg_def} requires that the solution provided by  the regularization operator $R_\alpha(y)$ tends to a constrained ML solution $\hat x$ as the mean of the Poisson variable tends to infinite. Hence the properties $2)$ and $2a)$ have to hold asymptotically.

As a first issue, we discuss the conditions when traditional stopping rules define a regularization algorithm in the sense of Definition \ref{stat_reg_def}. We consider the Morozov's discrepancy principle for the white Gaussian noise case and  three different rules for Poisson noise, i.e.
\begin{itemize}
 \item [i)] Adapted Morozov's discrepancy principle,
 \item [ii)] Pearson's cumulative test,
 \item [iii)] Poisson discrepancy criterion.
\end{itemize}
The most widely known and used criterion is the Morozov's discrepancy principle. It was first developed in the case of signal corrupted by Gaussian noise \cite{Morozov:105020}, but an adapted version can be restated for the Poisson case \cite{Bardsley,stagliano}. Pearson's cumulative test \cite{Veklerov,herbert} becomes the same as Morozov's discrepancy principle when the noise is white Gaussian, so it has been considered in the Poisson case. Poisson discrepancy criterion is a recently formulated stopping rule appropriate for Poisson noise \cite{bertero2010}. We will show now that only if $y \in H(\mathcal C)$ these stopping rules satisfy Definition \ref{stat_reg_def} with the consequence that, only when that condition is satisfied, ISRA and EM supplied with one of these stopping rules become well-defined regularization methods.

\subsection{Gaussian case}

Given the ISRA iterative process, Morozov's discrepancy principle says that a reliable estimate of the solution is obtained by choosing the first iteration $k$ such that
\begin{equation}
\| Hx^{(k)} -y \|^2 \leq \tau N \sigma^2 ~~~ ,
\label{Morozov_gauss}
\end{equation}
for some fixed $\tau>0$. For this rule the following proposition holds true.
\begin{thm}
\label{ISRA_Morozov_stop_rule}
ISRA supplied with the Morozov's discrepancy principle becomes a well-defined regularization method if and only if $y \in H(\mathcal C)$.
\end{thm}

\begin{proof}
We have to verify the second condition of Definition \ref{stat_reg_def}. When $\sigma = 0$, the stopping criterion is equivalent to the requirement that ISRA converges to the zero of the constrained LS functional. From Lemma \ref{ML_minimum} this follows if and only if $y \in H(\mathcal C)$.
\end{proof}

\subsection{Poisson case}

We first introduce two general properties of EM. The first one can be summarized as follows: EM produces scaled reconstructions when the input data are scaled. In fact,
\begin{lem}
Given a data $y$, let us consider a scalar $L>0$ and the corresponding scaled data $y_L=Ly$. Let $x^{(k)}$ indicate the $k$-th EM iteration with entry data $y$ as in \eref{em_alg}, and $x_L^{(k)}$ indicates the $k$-th EM iteration with entry data $y_L$. The following relation holds true
\begin{equation}
x_L^{(k)} = L x^{(k)} ~~~ .
\end{equation}
\label{em_scaling}
\end{lem}
\begin{proof}
Let $x^{(0)}$ the algorithm initialization. At the first iteration, with entry data $y_L$, we have
\begin{equation}
x_L^{(1)} = \frac{ x^{(0)} }{H^T 1} H^T \frac{Ly}{Hx^{(0)}} = L x^{(1)} ~~~ .
\end{equation}
From the second iteration onwards ($k \geq 1$) we have
\begin{equation}
x_L^{(k+1)} = \frac{ Lx^{(k)} }{H^T 1} H^T \frac{Ly}{L Hx^{(k)}} = L x^{(k+1)} ~~~ ,
\end{equation}
and hence the thesis holds true.
\end{proof}
We will also use the following well-known property of EM:
\begin{lem}
\label{flux}
For each iteration $k$ the relation  
\begin{equation}
\sum_{i=1}^N (Hx^{(k)})_i = \sum_{i=1}^N y_i 
\end{equation}
holds.
\end{lem}

\begin{proof}
The thesis follows directly from computations
\begin{eqnarray*}
\left< Hx^{(k)} ~,~ 1\right> &=& \left< \frac{ x^{(k-1)} }{H^T 1} H^T \frac{y}{Hx^{(k-1)}} ~,~ H^T 1 \right> \\
&=& \left<  \frac{y}{Hx^{(k-1)}} ~,~ Hx^{(k-1)} \right> = \left< ~ y ~,~ 1 ~\right>  ~~~ .
\end{eqnarray*}
\end{proof}

In the Poisson framework, the adapted version of Morozov's discrepancy principle says that a reliable estimate of the solution is obtained by choosing the first iteration $k$ such that
\begin{equation}
\| Hx^{(k)} -y \|^2 \leq \tau \sum_{i=1}^N (Hx^{(k)})_i ~~~,
\label{DP_poisson}
\end{equation}
where $x^{(k)}$ represents the EM iteration and $\tau$ is a fixed positive number.  

\begin{thm}
\label{EM_Morozov_stop_rule}
EM supplied with the adapted Morozov's discrepancy principle becomes a well-defined regularization method if and only if $y \in H(\mathcal C)$.
\end{thm}

\begin{proof}
We have to prove that the two conditions of Definition \ref{stat_reg_def} are satisfied. The first one is obvious. For the second one, we recall that
given a Poisson random variable with coefficient of variation $\delta$, for any realization $y_\delta$ the component $(y_\delta)_i$ tends to $\infty$
as $\delta \to 0$ and hence $\lim_{\delta \to 0} \sum_{i=1}^N (y_\delta)_i = \infty$. As a consequence of that, when $k$ can be written as a function of some parameter proportional to $\sum_{i=1}^N (y_\delta)_i$, i.e. $k=k(L)$, with $L= \gamma \sum_{i=1}^N (y_\delta)_i$, $\gamma >0$, we can rewrite condition $2a)$ in Definition \ref{stat_reg_def} as
\begin{equation}
\label{stop_rule_cond}
\lim_{L \to \infty} k(L) = \infty ~~~ .
\end{equation}

Now, in order to prove the thesis we will show that $k$ can be written as $k(L)$ and we will verify condition \eref{stop_rule_cond}. Let $x^{(k)}$ be the $k$-th EM iteration \eref{em_alg} with data entry $y$ and let $x_L^{(k)}$ be the $k$-th EM iteration with data entry $y_L=Ly$ with $L>0$. The adapted Morozov's discrepancy principle for $x_L^{(k)}$ consists in finding the first iteration $k$ such that
\begin{equation}
\| Hx_L^{(k)} -y_L \|^2 \leq \tau \sum_{i=1}^N (y_L)_i ~~~ .
\end{equation}
This relation can be written in terms of $y$ and the corresponding EM iteration $x^{(k)}$ using Lemma \ref{em_scaling} as
\begin{equation}
\| Hx^{(k)} -y \|^2 \leq \frac{\tau}{L} \sum_{i=1}^N y_i ~~~ .
\end{equation}
For $L \to \infty$ we have
\begin{equation}
\| Hx^{(k)} -y \|^2 \leq 0 ~~~ .
\label{morozov_infinity}
\end{equation}
From Lemma \ref{ML_minimum} and the convergence property of EM, the l.h.s. of \eref{morozov_infinity} converges to zero as $k$ tends to $\infty$ if and only if $y \in H(\mathcal{C})$. Therefore condition \eref{stop_rule_cond} is satisfied and the thesis is proved.
\end{proof}

Given the EM iterative process, Pearson's cumulative test says that a reliable estimate of the solution is obtained by choosing the first iteration $k$ such that
\begin{equation}
\sum_{i=1}^N \frac{(Hx-y)^2_i}{(Hx)_i} \leq \tau~N~~~ ,
\label{chi_square}
\end{equation}
where $N$ represents the number of data points, as discussed in \cite{puetter}, and $\tau$ is a fixed positive number.

\begin{thm}
\label{EM_Pearson_stop_rule}
EM supplied with the Pearson's cumulative test becomes a well-defined regularization method
if and only if $y \in H(\mathcal C)$.
\end{thm}

\begin{proof}
Let $x^{(k)}$ be the EM iteration with data entry $y$. 
By taking $y_L = Ly$ with $L>0$ the stopping criterion for EM applied
to the data $y_L$ takes the form
\begin{equation}
\sum_{i=1}^N \frac{(Hx^{(k)}-y)^2_i}{(Hx^{(k)})_i} \leq \tau~\frac{N}{L}
\end{equation}
and the thesis follows as in Theorem \ref{EM_Morozov_stop_rule}.
\end{proof}

Given the EM iterative process, Poisson discrepancy criterion says that
a reliable estimate of the solution is obtained by choosing the first iteration $k$ such that
\begin{equation}
\frac{2}{N} \sum_{i=1}^N  y_i \log \frac{y_i}{(Hx^{(k)})_i} + (Hx^{(k)})_i - y_i \leq \tau~~~,
\end{equation}
for some positive number $\tau$. Poisson discrepancy criterion with $\tau=1$ has been proposed both for choosing the regularization parameter in an EM scheme with penalty term and for stopping the unpenalized EM iterative process \cite{bertero2010}.

\begin{thm}
\label{EM_Poisson_stop_rule}
EM supplied with the Poisson discrepancy criterion becomes a well-defined regularization method if and only if $y \in H(\mathcal C)$.
\end{thm}

\begin{proof}
Let $x^{(k)}$ be the EM iteration with data entry $y$. 
By taking $y_L = Ly$, with $L>0$, the stopping criterion for EM applied to the data $y_L$
takes the form
\begin{equation}
\frac{2}{N}\sum_{i=1}^N  y_i \log \frac{y_i}{(Hx^{(k)})_i} + (Hx^{(k)})_i - y_i \leq \frac{\tau}{L}
\end{equation}
and the thesis follows as in Theorem \ref{EM_Morozov_stop_rule}.
\end{proof}

These theorems point out that condition $y \in H(\mathcal C)$ is crucial
for these classical stopping rules to define a regularization algorithm for ISRA and EM.
Tikhonov defines the cases where this property is not satisfied
as genuinely ill-posed.
In practice, what very often happens is that data
do not belong to the convex cone $H(\mathcal C)$.
Indeed, if the true object $x_T$ contains even one zero value,
the corresponding signal $y_T=Hx_T$ belongs to the frontier of $H(\mathcal C)$,
and any little variation due to noise fluctuation
can move the data $y$ outside $H(\mathcal C)$.
Moreover, in such cases the standard stopping rules could not work.
We will see an example in \Sref{sec:Experiments}.

In the next Section we will introduce a stopping criterion defining a regularization
algorithm for constrained ML problems that is valid even in the case of genuinely ill-posed problems.

\section{Constrained stopping criterion}
\label{sec:Criterion}

Our aim is to realize a stopping rule which gives rise
to a regularization algorithm for every $y \in \R^N$.
To do this we first observe that the regularization methods discussed in
the previous section are based on equations of the kind
\begin{equation}
\label{generalized_morozov}
r(x^{(k)},y) \leq \tau~\E \left( r(x^{(k)},y) \right) ~~~ ,
\end{equation}
where
$r(x,y)$ is a function specific for each rule and $\tau$ is a fixed positive number.
Choosing $r(x,y)=D_{LS}((x,y))$ we find the Morozov's discrepancy principle;
choosing $r(x,y)=\|(Hx-y)/\sqrt{Hx}\|^2$ we find the Pearson's test criterion;
choosing $r(x,y)=D_{KL}(y,x)$ we find the Poisson discrepancy principle.

The main drawbacks of these criteria is that for genuinely ill-posed problems,
i.e. when $y \not\in H(\mathcal C)$, the l.h.s. of \eref{generalized_morozov}
does not converge to $0$ for $k \to \infty$ (see Lemma \ref{zero_conv}), and therefore
it may happen that the stopping rule is never applied.
A choice of $r(x,y)$ that for sure overtakes this difficulty
is a function converging to $0$ for $k \to \infty$ and whose expected value is positive.
In this paper we propose to choose
\begin{equation}
\label{CBR_definition}
r(x,y) = \|x ~ \nabla L_y(x)\|^2 ~~~ .
\end{equation}
In fact, this choice ensures that given an algorithm $x^{(k)}$
converging to the constrained minimum of \eref{neglog_problem}, then $\lim_{k \to \infty} r(x^{(k)},y) = 0$.
Let us call $r(x,y)$ the \emph{constrained backprojected residual} (CBR)
and therefore the associated stopping rule \eref{generalized_morozov} is the CBR criterion.
In the following we describe the CBR criterion for the Gaussian and Poisson cases
and prove that, if applied to ISRA and EM,
the stopping rule leads to two regularization algorithms in the sense of Definition \ref{stat_reg_def},
without any restriction on the input data.
In the next section, we will also show that the CBR criterion works in a very
reliable way, using applications to astronomical image processing.

\subsection{Gaussian case}

In the Gaussian case the CBR takes the form
\begin{equation}
r(x,y) =  \| x ~ H^T \left( Hx - y \right) \|^2 ~~~ ,
\end{equation}
and the following result holds true:
\begin{prp}
In the Gaussian case, the expected value of $r(x,y)$ is
\begin{equation}
\E (r(x,y)) =  \sum_{j=1}^M x_j^2 ~ (H_2^T \sigma^2)_j ~~~ ,
\end{equation}
where $(H_2)_{ij} = (H_{ij})^2$.
\end{prp}

\begin{proof}
The expected value of each component is
\begin{equation}
\E \left[ x^2_j ~ ( H^T \left( Hx - y \right) )^2_j \right] =  x^2_j ~ \E \left[ ( H^T \left( Hx - y \right) )^2_j \right] ~~~.
\end{equation}
Since the noise is independently distributed,
the expected value of the product $(Hx-y)_i (Hx-y)_h$  
is different from zero only if $i=h$. Then
\begin{equation}
\E \left[ x^2_j ~ ( H^T \left( Hx - y \right) )^2_j \right] =
x^2_j ~ (H_2^T ~ \sigma^2)_j ~~~ ,
\end{equation}
since $\E \left[ \left( Hx - y \right)_i^2 \right] = \sigma^2$ for all $i$.
By summing up all the components we get the thesis.
\end{proof}

Now, we can explicitly write the CBR criterion for the Gaussian case.
Giving an algorithm $x^{(k)}$ converging to a constrained LS solution,
the criterion says that a reliable estimate of the solution is obtained
by choosing the first iteration $k$ such that:
\begin{equation}
\| x^{(k)} ~ H^T \left( Hx^{(k)} - y \right) \|^2 \leq \tau~\sum_{i=1}^N (x_i^{(k)})^2  (H_2^T \sigma^2)_i ~~~ ,
\label{gauss_cdc}
\end{equation}
with $\tau$ fixed positive number.
\begin{thm}
ISRA supplied with the CBR criterion (\ref{gauss_cdc}),
becomes a well-defined regularization method.
\label{ISRA_CBR}
\end{thm}

\begin{proof}
The first condition of Definition \ref{stat_reg_def} is obvious.
As for the second condition, in the Gaussian case the coefficient of variation $\delta$ tends to $0$ if and only if the standard deviation $\sigma$ tends to $0$.
Therefore to compute the limit in the second condition of Definition \ref{stat_reg_def} 
we can set $\sigma=0$ and so the CBR criterion (\ref{gauss_cdc})
for $\delta \to 0$ becomes
\begin{equation}
\| ~ x^{(k)} ~ H^T \left( Hx^{(k)} - y \right) \|^2 \leq 0 ~~~ .
\label{isra_cond}
\end{equation}
Since the algorithm is convergent, the l.h.s. in \eref{isra_cond} converges to $0$. 
\end{proof}

\subsection{Poisson case}

In the Poisson case, the CBR takes the form
\begin{equation}
r(x,y) =  \left\| x ~ H^T \left( 1 - \frac{y}{Hx} \right) \right\|^2 ~~~ ,
\end{equation}
and the following result holds true:
\begin{prp}
In the Poisson case, the expected value of $r(x,y)$ is
\begin{equation}
\E (r(x,y)) =  \sum_{j=1}^M x_j^2 ~ \left( H_2^T \frac{1}{Hx} \right)_j ~~~ ,
\end{equation}
where $(H_2)_{ij} = (H_{ij})^2$.
\end{prp}
\begin{proof}
The expected value of each component is
\begin{eqnarray*}
\E \left[ x^2_j ~ \left( H^T \left( 1 - \frac{y}{Hx} \right) \right)_j^2 \right] &=& 
x^2_j ~ \E \left[ \left( H^T \left( 1 - \frac{y}{Hx} \right) \right)_j^2 \right] \\
&=& x^2_j ~ \E \left[ \left( H^T \left( \frac{1}{Hx} ( Hx - y ) \right) \right)_j^2 \right] ~~~ .
\end{eqnarray*}
As in the Gaussian case, noise is independently distributed. Therefore we can exploit again the fact that  the expected value of the product $(Hx-y)_i (Hx-y)_h$ is different from zero only if $i=h$. Then
\begin{equation}
\E \left[ x^2_j ~ \left( H^T \left( 1 - \frac{y}{Hx} \right) \right)_j^2 \right] =
x^2_j ~ \left( H_2^T \left( \frac{1}{(Hx)^2} \E \left[ ( Hx - y )^2 \right] \right) \right)_j ~~~.
\end{equation}
Since noise is Poisson, $\E \left[ ( Hx - y )^2 \right] = Hx$. Summing up all the components we have the thesis.
\end{proof}

Now, we can explicitly write the CBR criterion for the Poisson case. The criterion says that a reliable estimate of the solution is obtained by choosing the first iteration $k$ such that
\begin{equation}
\left\| ~ x^{(k)}  ~ H^T \left( 1-\frac{y}{Hx^{(k)}} \right) \right\|^2 \leq \tau~\sum_{i=1}^N (x^{(k)})_i^2 \left( H_2^T \frac{1}{Hx^{(k)}} \right)_i ~~~,
\label{cdc_poisson}
\end{equation}
where $x^{(k)}$ is the EM iteration and $\tau$ is a fixed positive number. To prove the regularization property for this stopping rule applied to EM, we need the following Lemma.
\begin{lem}
\label{limited}
Let $x^{(k)}$ be the $k$-th iteration of EM. The following inequality holds:
\begin{equation}
\sum_{i=j}^M (x^{(k)})_j^2 \left( H_2^T \frac{1}{Hx^{(k)}} \right)_j \leq \sum_{i=1}^N y_i ~.
\end{equation}
\end{lem}

\begin{proof}
For simplicity of notation we use $x$ instead of $x^{(k)}$.
We begin by noting that
\begin{equation}
\sum_{i=1}^M x_j^2 \left( H_2^T \frac{1}{Hx} \right)_j  = 
\sum_{i=1}^N (H_2 x^2)_i \left(  \frac{1}{Hx} \right)_i ~~~ .
\end{equation}
Then
\begin{eqnarray*}
\sum_{i=1}^N \frac{ \sum_{j=1}^M h^2_{ij} x^2_j }{ \sum_{j=1}^M h_{ij} x_j } \leq
\sum_{i=1}^N \sqrt{ \sum_{j=1}^M h^2_{ij} x^2_j } \leq
\sum_{i=1}^N \sum_{j=1}^M h_{ij} x_j \leq
\sum_{i=1}^N y_i
\end{eqnarray*}
having used the relation
$\sqrt{ \sum_{j=1}^M h^2_{ij} x^2_j } \leq \sum_{j=1}^M h_{ij} x_j$
and Lemma \ref{flux}.
\end{proof}

\begin{thm}
EM supplied with the CBR criterion (\ref{cdc_poisson}) becomes a well-defined regularization method.
\label{EM_CBR}
\end{thm}

\begin{proof}
The first condition of Definition \ref{stat_reg_def} is obvious.
For the second one, 
we consider a given data $y$ and the data $y_L=Ly$ scaled by a number $L>0$.
We will prove that when $L$ tends to infinity
the number of iterations needed to satisfy the CBR criterion tends to infinity.
According to the remark at the beginning of the proof of Theorem \ref{EM_Morozov_stop_rule},
this prove the condition $2a)$ of the Definition \ref{stat_reg_def}.

Given the data $y_L$ and the corresponding EM algorithm $x_L^{(k)}$
the CBR criterion stops the iterative process  at the first iteration $k$ such that
\begin{equation}
\left\| ~ x_L^{(k)} ~ H^T \left( 1-\frac{y_L}{Hx_L^{(k)}} \right) \right\|^2 \leq \tau~\sum_{i=1}^M (x_L^{(k)})_j^2 \left( H_2^T \frac{1}{Hx_L^{(k)}} \right)_j  ~~~ ,
\end{equation}
for a fixed $\tau>0$. This equation can be written by replacing $y_L$ with $y$ 
and the corresponding EM algorithm $x_L^{(k)}$ with $x^{(k)}$ by using Lemma \ref{em_scaling}
and picking up the $L$ factor, i.e.
\begin{equation}
\left\| ~ x^{(k)} ~ H^T \left( 1-\frac{y}{Hx^{(k)}} \right) \right\|^2 \leq 
\frac{\tau}{L} \sum_{i=1}^M (x^{(k)})_j^2 \left( H_2^T \frac{1}{Hx^{(k)}} \right)_j  \leq 
\frac{\tau}{L} \sum_{i=1}^N y_i ~,
\end{equation}
where the second inequality holds by Lemma \ref{limited}.
For $L \to \infty$ we have
\begin{equation}
\left\| ~ x^{(k)} ~ H^T \left( 1-\frac{y}{Hx^{(k)}} \right) \right\|^2 \leq 0 ~~~ .
\label{em_cond}
\end{equation}
Since the algorithm is convergent, the l.h.s. in \eref{em_cond} converges to $0$. 
\end{proof}

\section{Numerical experiments}
\label{sec:Experiments}

In this section we test the proposed regularization algorithm in
the case of image reconstruction from data recorded by a solar hard X-ray satellite.
The Reuven Ramaty High Energy Solar Spectroscopic Imager (RHESSI) \cite{lin} mission
has been launched by NASA in February 2002 with the aim of investigating
emission and energy transport mechanisms during solar flares.
RHESSI hardware is made of nine pairs of rotating collimators
that time-modulate the incoming photon flux before it is detected by the corresponding
Ge detectors. As a consequence, the RHESSI 
imaging problem consists in locally retrieving the photon flux intensity image
starting from a given set of count modulation profiles.

We first validate the stopping rule introduced in this paper
for ISRA and EM, in the case of synthetic data mimicking the modulation introduced by RHESSI grids and affected by Gaussian and Poisson noise, respectively.
Then we investigate the behavior of EM equipped with the new criterion
when applied to a set of experimental RHESSI observations.

\subsection{Simulated data}

Solar flares are sudden and intense explosions occurring high in the solar corona
and accelerating electrons down to the thicker chromospheric part of the solar atmosphere.
While diving into the plasma driven by the flow lines of intense magnetic fields,
these electrons emit hard X-rays by collisional bremsstrahlung. As a result, typical hard X-ray source
configurations observed by RHESSI present one, two or a higher number of bright footpoints on a weak background.
In this example we consider a simulated source configuration mimicking
the structure and the physical properties of a real flare
observed by RHESSI on July 23 2002 (Figure \ref{fig:true_image}).
Such configuration is contained in a $64$ by $64$ pixel image.
The sources have Gaussian form and are arranged according to the following scheme:
the left source (source L) is located at coordinates $(16,32)$ with variance $0.64$ and amplitude $1.28$,
the center source (source C) is located at coordinates $(32,32)$ with variance $0.64$ and amplitude $1.6$,
the upper right source (source UR) is located at coordinates $(42,45)$ with variance $0.48$ and amplitude $0.6$ and
the lower right source (source LR) is located at coordinates $(42,19)$ with variance $0.64$ and amplitude $1.28$.
Using the routines of Solar SoftWare (SSW) \href{http://www.mssl.ucl.ac.uk/surf/sswdoc/}{\it http://www.mssl.ucl.ac.uk/surf/sswdoc/} we reproduced the 
RHESSI acquisition process when the grids are reached by the photon flux emitted by such simulated source
constellation and simulated two different sets of count modulation profiles using detectors from 3 through 8.

The first set is affected by white Gaussian noise with standard deviation $\sigma=10$.
\begin{figure}
\begin{center}
\includegraphics[scale=0.5]{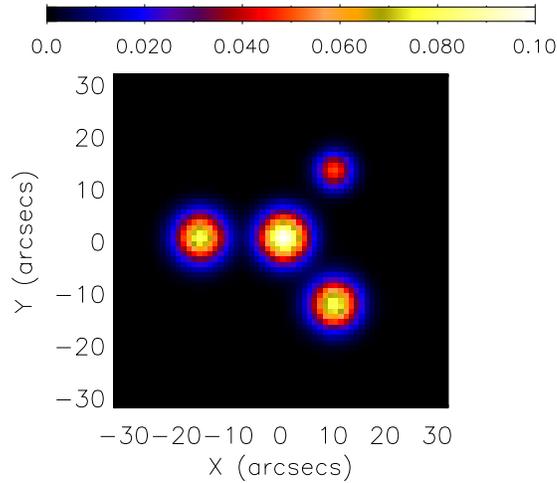}
\end{center}
\caption{Simulated source configuration mimicking the structure of a real flare
observed by RHESSI on July 23 2002}
\label{fig:true_image}
\end{figure}
We applied ISRA against these synthetic data and we stopped the iterations both in correspondence with the minimum of the L2-norm of the difference between the reconstructed and the simulated configuration (L2-norm criteria) and by means of the CBR criterion, using $\tau = 1 / \sigma^2$ (such choice has been made to contrast the typical oversmoothing property of discrepancy methods). L2-norm minimization applies at the 518-th iteration and CBR criterion at the 645-th iteration.
Figure \ref{fig:isra_reconstruction} shows the images obtained with the two stopping rules. The shapes of the sources determined using the two methods are very similar
and are also consistent with the Gaussian source structures assumed in the original configuration. On the other hand the image provided by the CBR rule shows a better separation for the four sources. We used the images obtained by the two stopping rules to
compute the total flux emitted by the four sources within a square with side length $13$ pixels and located around the center of the sources of the original configuration.
Then we computed the rate of the flux reconstructed by the two criteria for each sources.
Table \ref{tab:isra_cfr} shows that for source C the difference in reconstructing the original photometry between CBR criterion and L2-minimization is negligible.
Conversely, the CBR rule has a better photometry in the case of the other three sources and, particularly,
in the case of source UR, which is the weakest one in the simulated configuration.
\begin{table}
\begin{center}
\begin{tabular}{c|ccccc}
Source position & flux & L2-norm & CBR & L2-norm \%  & CBR \% \\ \hline
C & 5.884   &    4.025  &  4.057 &  68.4 & 68.9 \\
L &   4.691   &    3.882  &  4.048 &  82.7 & 86.2 \\     
BR &  4.762   &  4.265 & 4.456  &  89.5 & 93.5\\
UR &  1.824   &  1.454  &  1.654 & 79.7 & 90.6
\end{tabular}
\end{center}
\caption{Comparison between the photometry of the solutions
provided by ISRA stopped using the CBR and the L2-norm.
The first column represents the original flux integrated in a $13$ by $13$ square
located around the center of the sources. The second and the third columns
indicate the fluxes reconstructed by the two methods within the same square,
and the fourth and fifth columns show the rates of the reconstructed fluxes.}
\label{tab:isra_cfr}
\end{table}
\begin{figure}
\includegraphics[scale=0.45]{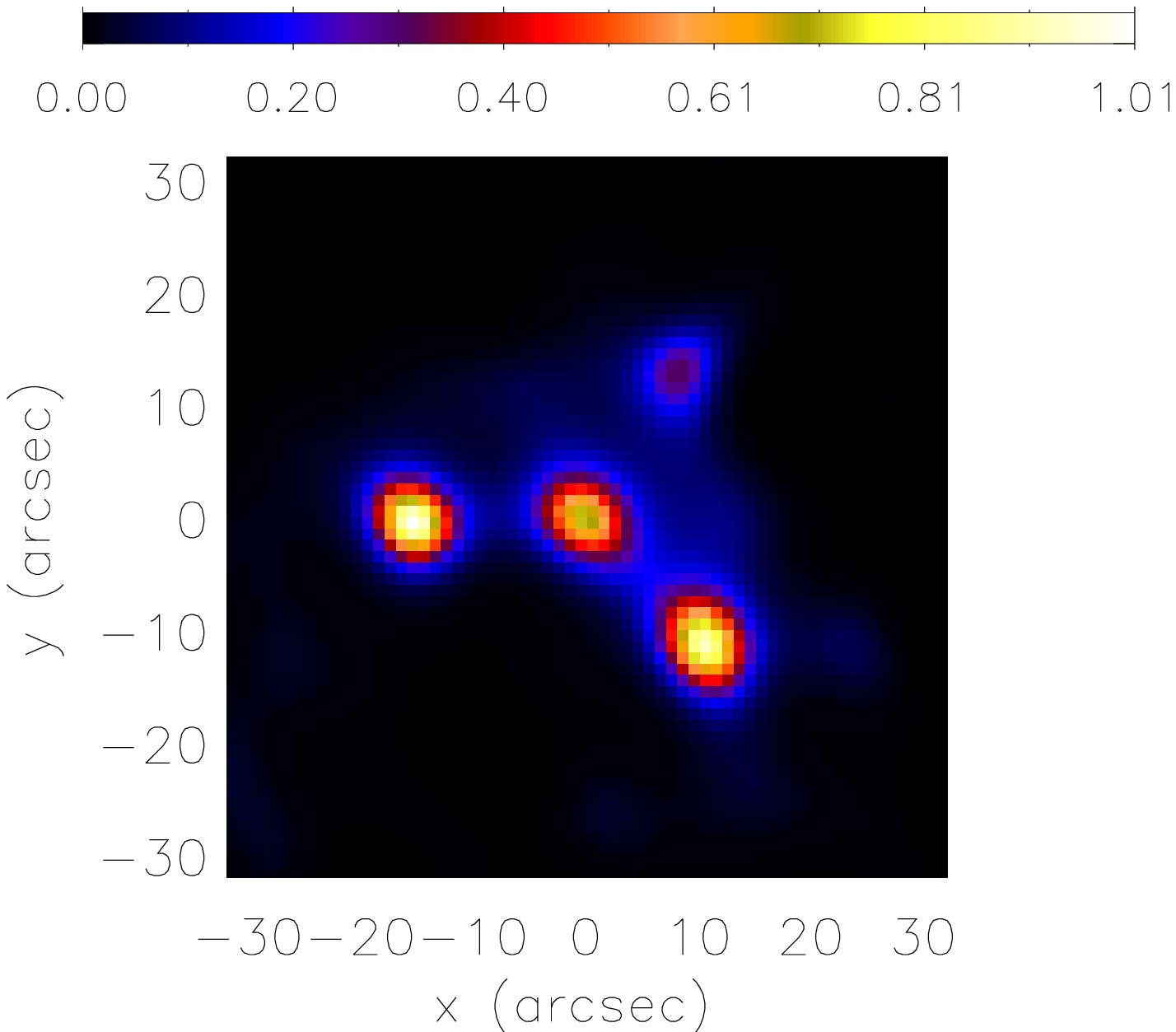}
\includegraphics[scale=0.45]{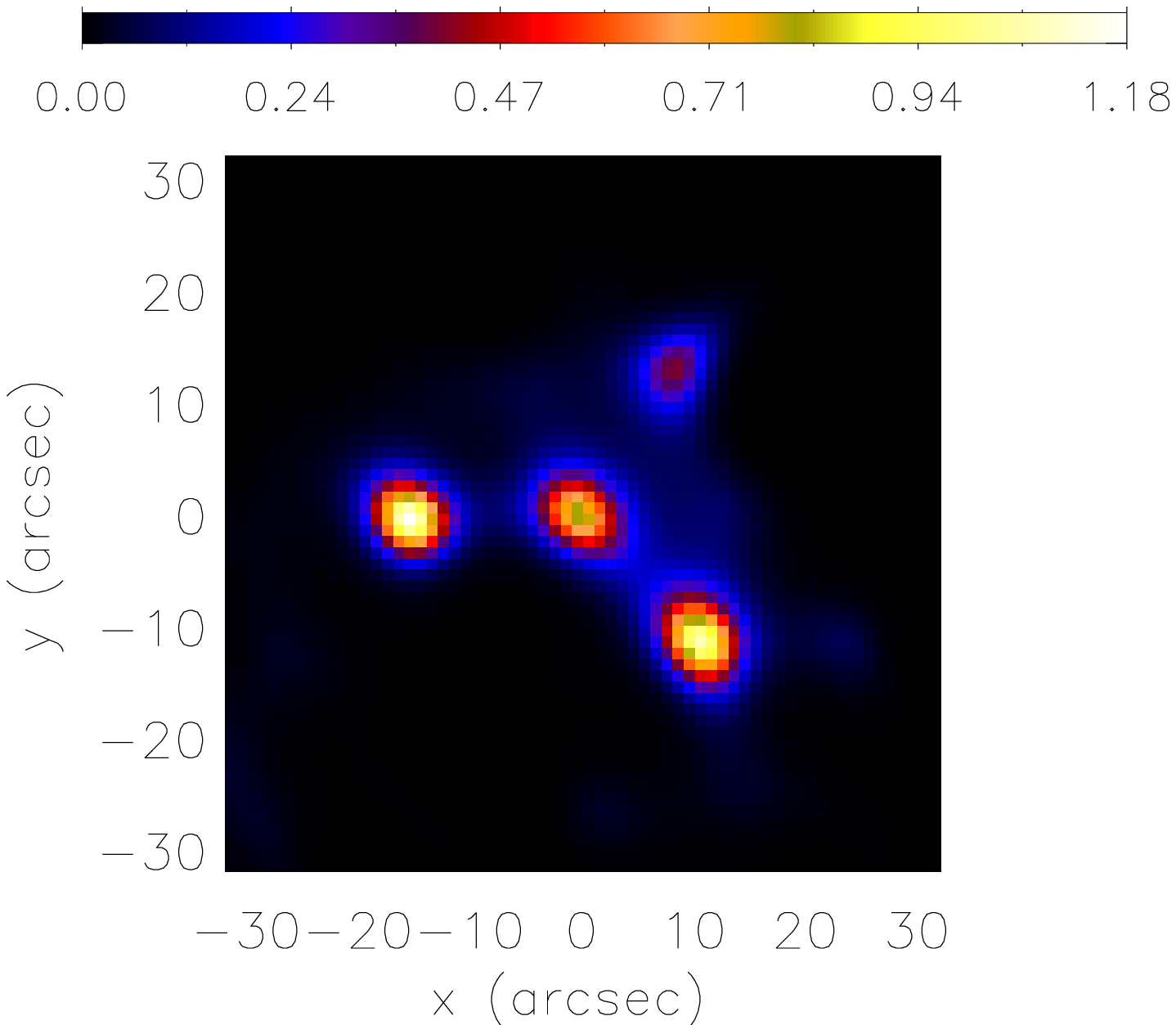}
\caption{From left to right: the reconstruction with ISRA stopped in correspondence with the minimum of L2-norm (518-th iteration) and the reconstruction with ISRA stopped according to CBR criterion (645-th iteration).}
\label{fig:isra_reconstruction}
\end{figure}

The second set of synthetic profiles has been affected by Poisson noise and EM was applied to obtain
images in Figure \ref{fig:em_reconstruction}. The left panel of Figure \ref{fig:em_reconstruction}
shows the image obtained stopping the iterations using the L2-norm, and the right panel
shows the image obtained using the CBR (coherently with the Gaussian case, here we chose $\tau=N/\sum_{i=1}^N y_i$).
The L2-norm rule stops EM at the 133-th iteration and the CBR criterion at the 229-th iteration.
Both reconstructions show that source UR has a smaller support with respect to what obtained by ISRA in the case of white Gaussian noise.
This is a direct consequence of the fact that Poisson noise is signal dependent and hence the weakest source has smaller signal to noise ratio with respect to the other sources. However, Table \ref{tab:em_cfr} shows that for all sources and specifically for source UR, the CBR criterion reconstructs the photometry in a much more accurate way.
\begin{table}
\begin{center}
\begin{tabular}{c|ccccc}
Source position & flux & L2-norm & CBR & L2-norm \% & CBR \% \\ \hline
C & 5.884 &      4.391  &  4.484 &  74.6 & 76.2 \\
L &  4.691  &     3.725  &  4.296 &  79.4 & 91.5 \\     
BR &  4.762   &  3.857 & 4.405  &   80.9 & 92.5\\
UR  &  1.824   &   0.838  &  1.128 & 45.9 & 61.8
\end{tabular}
\end{center}
\caption{Comparison between the photometry of the solutions
provided by EM with the CBR criterion and the L2-norm stopping.
The columns are arranged according to the same scheme of table \ref{tab:isra_cfr}.}
\label{tab:em_cfr}
\end{table}
\begin{figure}
\includegraphics[scale=0.45]{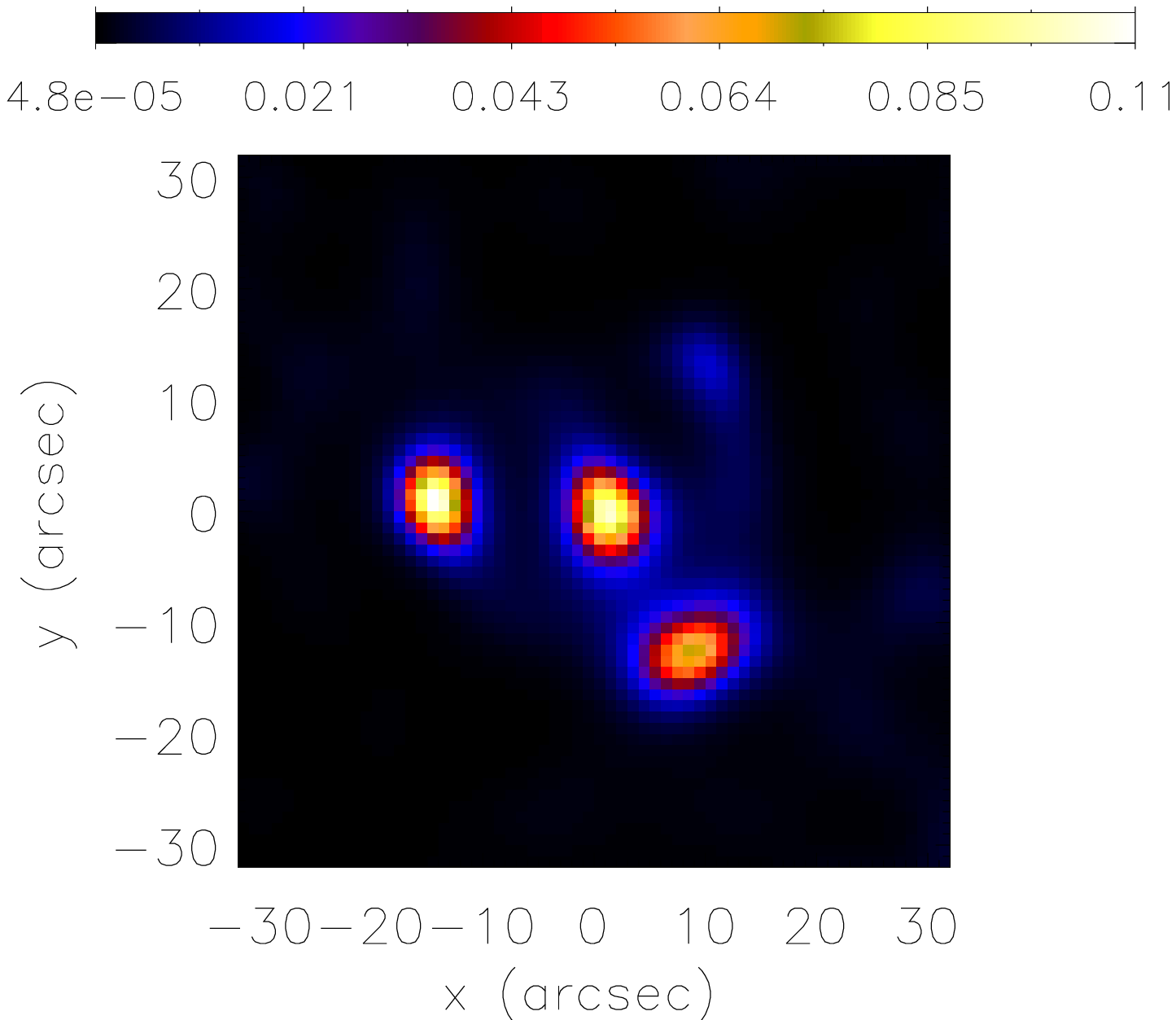}
\includegraphics[scale=0.45]{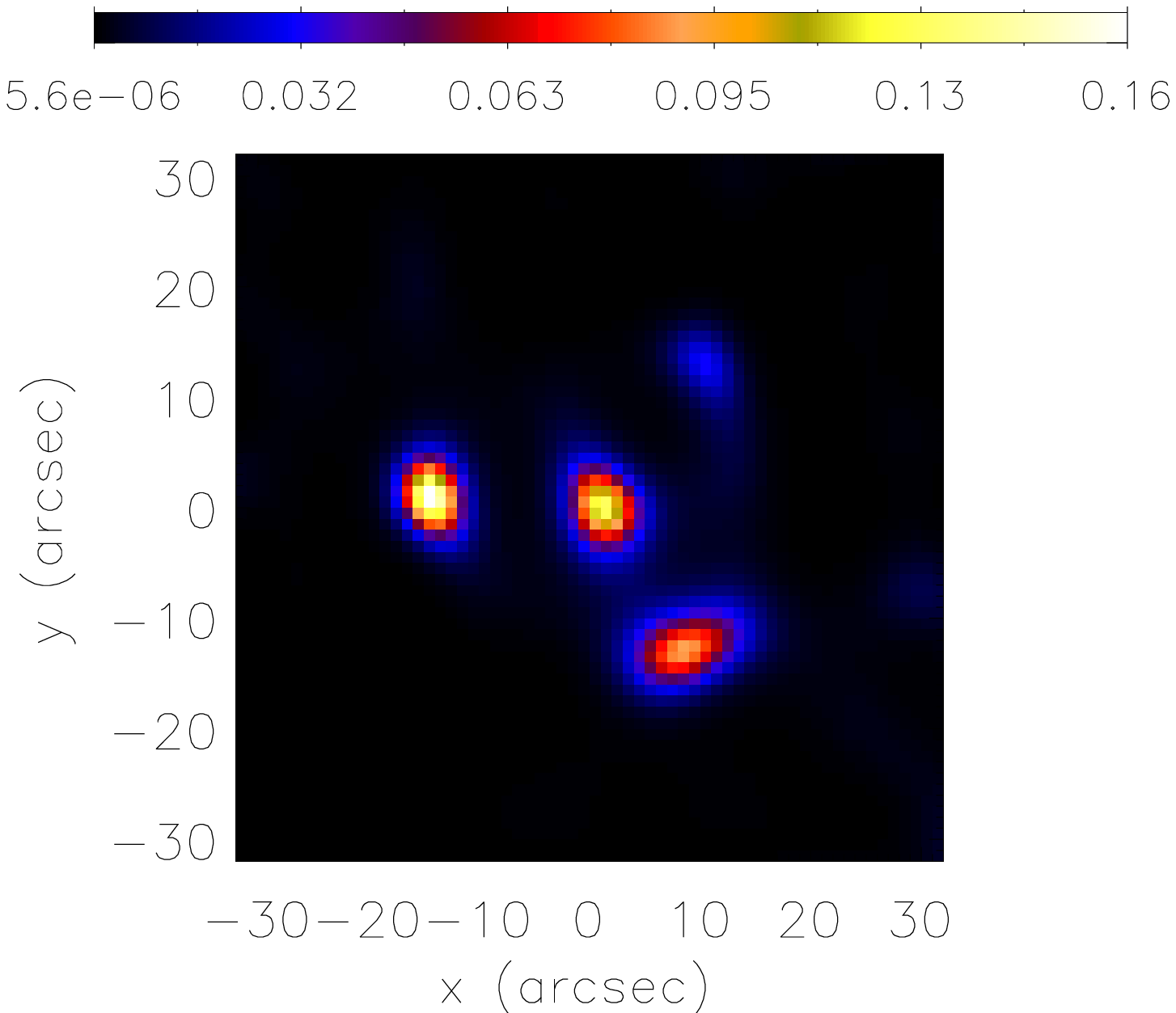}
\caption{From left to right: the reconstruction with EM
at the minimum of the 2-norm (133-th iteration) and the reconstruction with EM
stopped according to the proposed stopping rule (229-th iteration).
}
\label{fig:em_reconstruction}
\end{figure}

\subsection{Real data}

We then studied the behavior of EM regularized by the CBR criterion for the reconstruction of the photon flux map of two real flaring events.
The first event is the September 8 2002 flare in the time interval between 01:38:44 and 01:39:35 UT. The data have been collected by detectors 3 through 8, in the energy range between 25 and 30 keV. The second event is the November 3 2003 flare in the time interval between 01:32:42 and 01:42:25 UT. The data have been collected by detectors 3 through 8, in the energy range between 12 and 25 keV. During the first event the total number of counts collected is about $7.45~10^4$, the number of data is $N = 3816$. During the second event the total number of counts collected is about 
$1.38~10^6$ and the number of data is $N = 3168$. In both cases the reconstructed field of view is a square of $80$ arcseconds side length corresponding to a $64$ by $64$ pixel image.

For these two events EM regularized with the CBR criterion provides the two reconstruction in \Fref{fig:real_case_reconstruction}. In the left panel case, the stopping rule applies after 498 iterations while in the right pane case after 822 iterations. In order to compare the behaviors of the CBR criterion and of the other criteria for Poisson noise described in \Sref{fig:real_case_reconstruction}, in 
\Fref{fig:real_cases_criteria} we computed the discrepancy equations corresponding to the Morozov's discrepancy \eref{EM_Morozov_stop_rule}, Pearson's test \eref{EM_Pearson_stop_rule} and Poisson discrepancy \eref{EM_Poisson_stop_rule}. For five cases over six, the criteria never apply since the left hand sides of the discrepancy equations never intersect the corresponding right hand sides. There is just one case that works (Pearson's test on the reconstruction of the September 08 2002 event) but just after 1000 iterations and the corresponding reconstruction is clearly undersmoothed.
\begin{center}
\begin{figure}
\includegraphics[scale=.9]{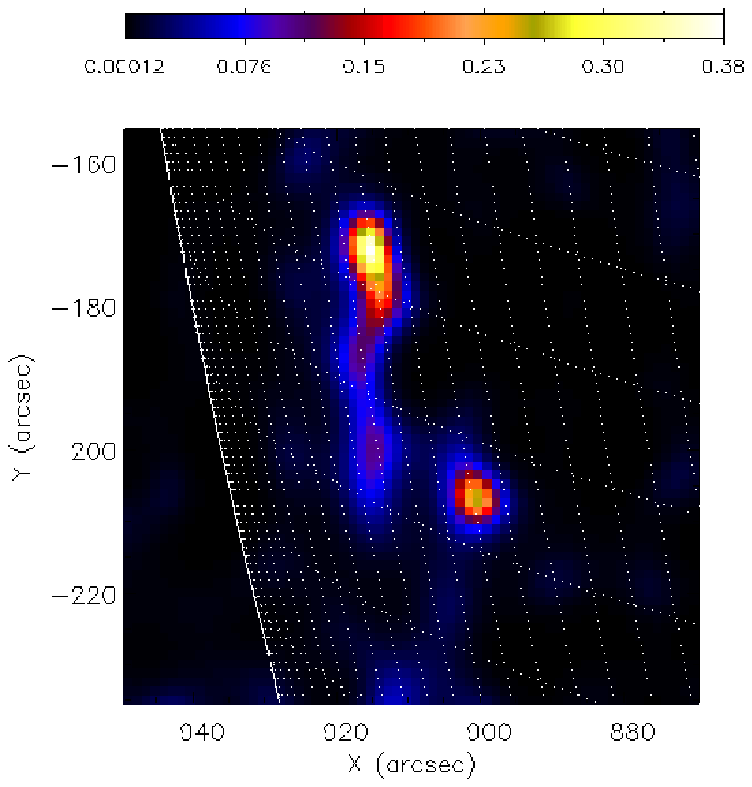}
\includegraphics[scale=.9]{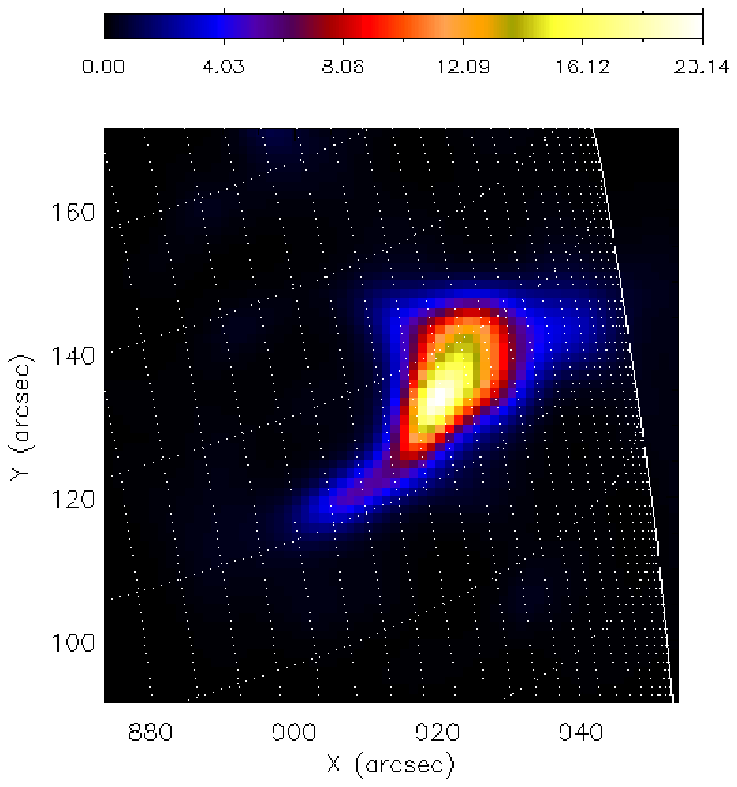}
\caption{Left panel: the reconstruction of the 08 September 2002 event 
performed with EM regularized by the CBR criterion (498 iterations).
Left panel: the reconstruction of the 03 November 2003 event 
performed with EM regularized by the CBR criterion (822 iterations).
In both cases the white grid represents the Sun's surface.}
\label{fig:real_case_reconstruction}
\end{figure}
\begin{figure}
\includegraphics[scale=0.8]{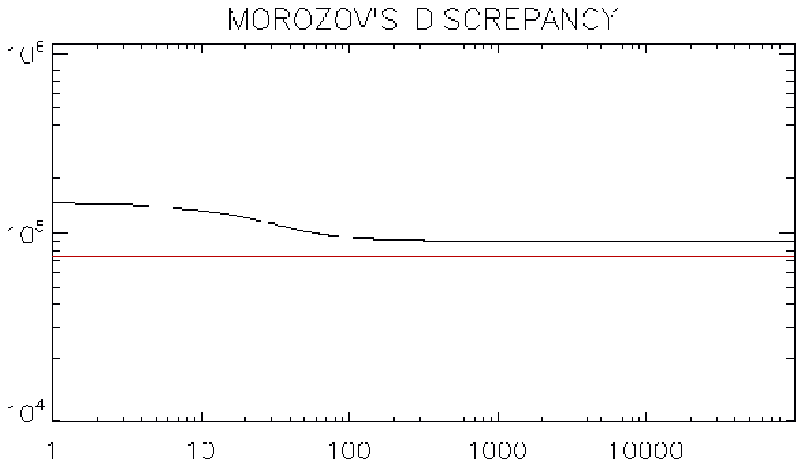}
\includegraphics[scale=0.8]{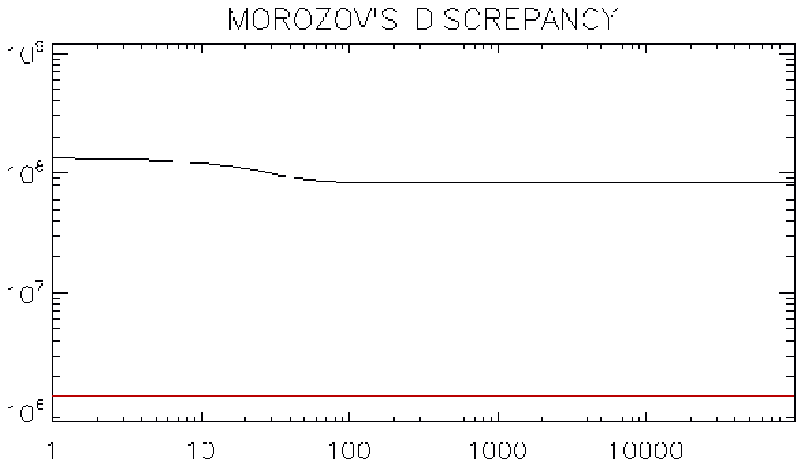}\\
\includegraphics[scale=0.8]{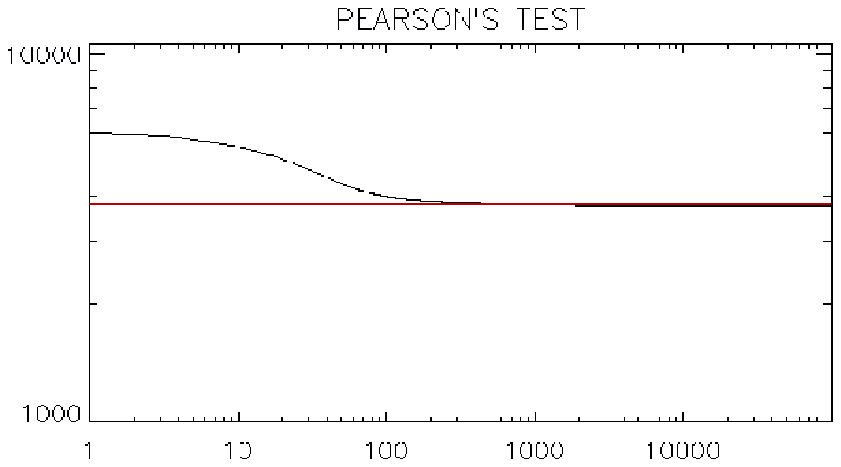}
\includegraphics[scale=0.8]{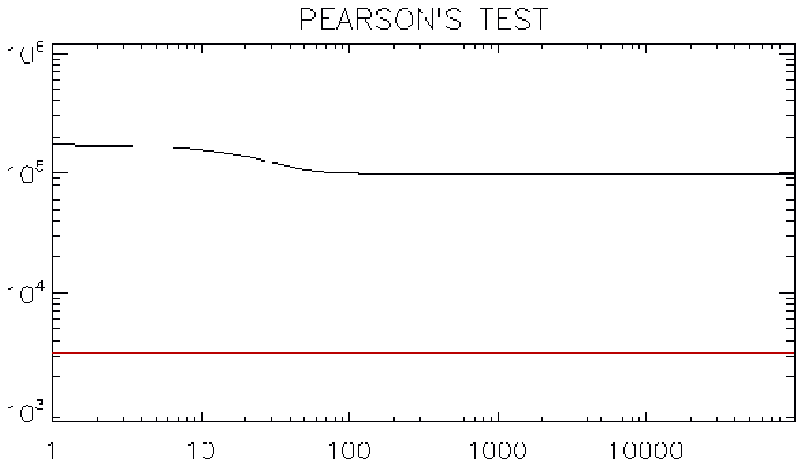}\\
\includegraphics[scale=0.315]{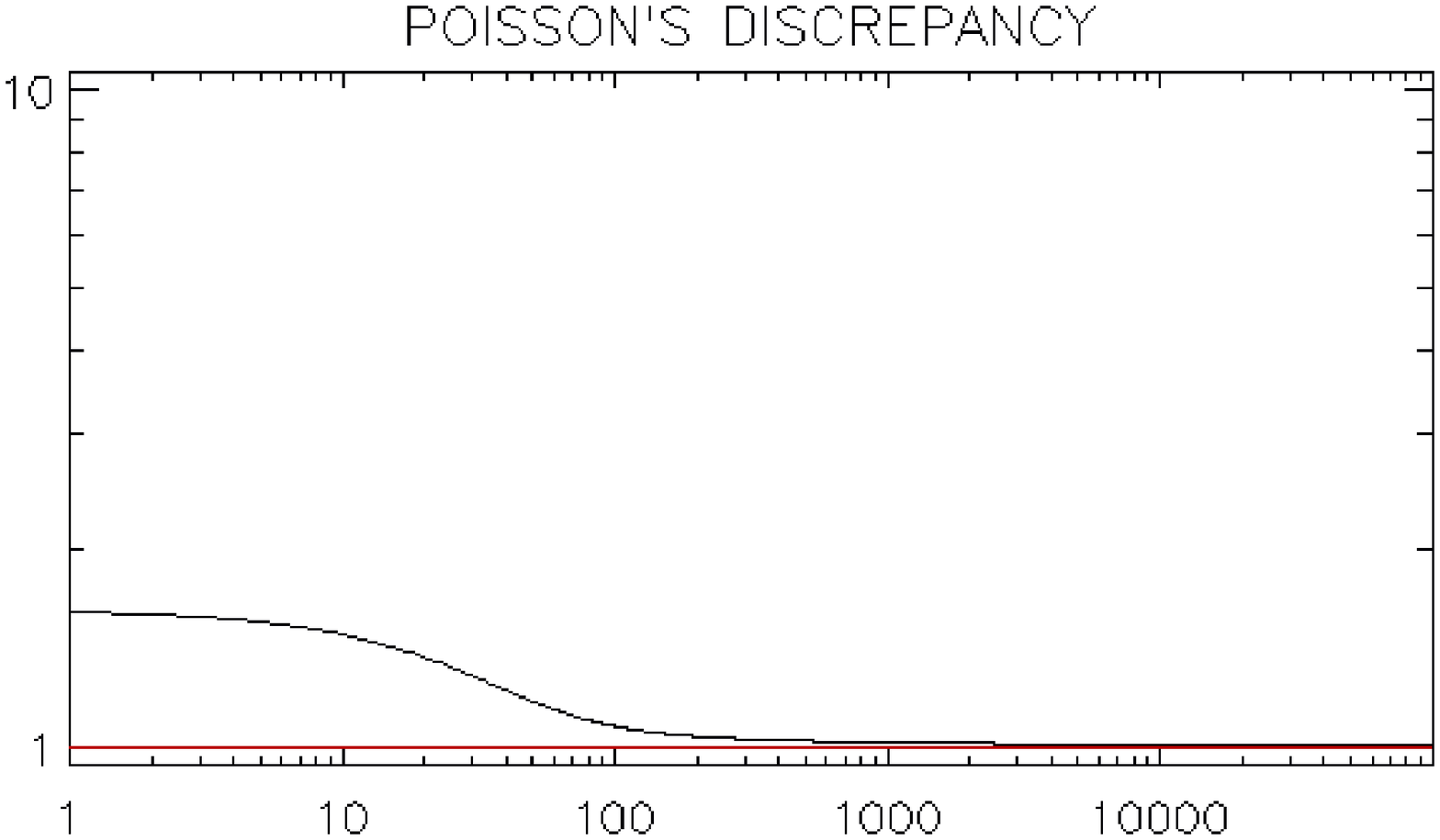}
\includegraphics[scale=0.315]{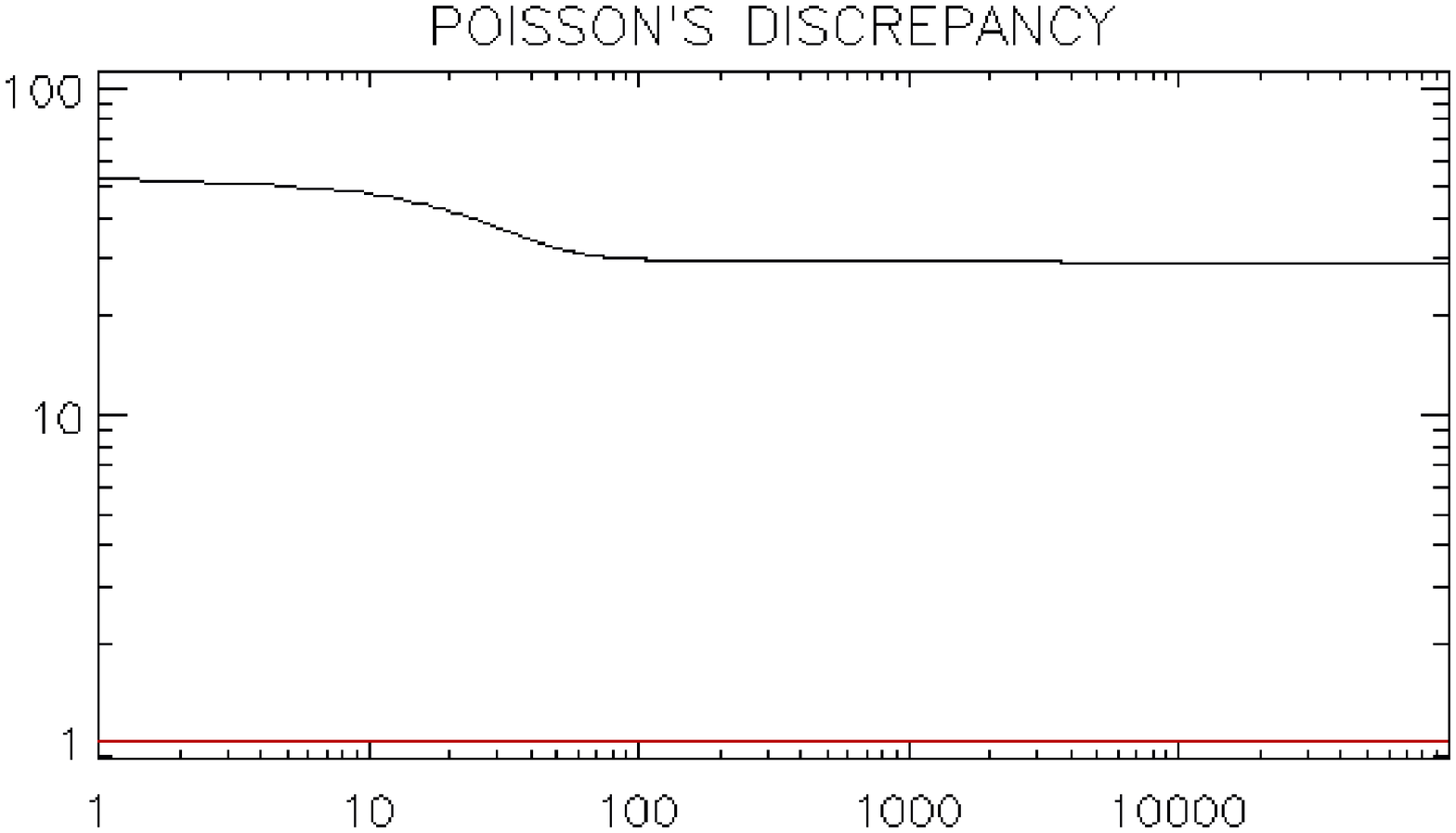}\\
\caption{From top to bottom: the l.h.s. of the classical stopping rules (equations \eref{EM_Morozov_stop_rule}, \eref{EM_Pearson_stop_rule} and \eref{EM_Poisson_stop_rule}) plotted against the number of iterations and visualized in black on a logarithmic scale. The constant r.h.s. values are in red. The first and the second columns show the stopping rules computed on the EM reconstructions of the September 08 2002 event and November 03 2003 events, respectively.}
\label{fig:real_cases_criteria}
\end{figure}
\end{center}

\section{Conclusions}

In the present paper we have formulated a general stopping criterion for constrained ML algorithms. This new criterion, called CBR criterion, is based on the statistical properties of the signal and takes into account the constraint on solution.
We have generalized the Tikhonov definition of regularization
for constrained ML problems and we have proved that ISRA and EM procedures,
equipped with the CBR criterion, are well-defined regularization algorithms.
Moreover, we have also proved that the traditional stopping rules,
applied to any convergent constrained ML algorithm, 
do not define regularization algorithms in the case of genuinely ill-posed problems.

We have illustrated the method first against synthetic count modulation profiles 
simulated in the framework of an X-ray solar mission. Specifically an analysis of photometry in specific region of interest showed the accuracy of this new stopping rule. We also considered two real observations and reconstructed the X-ray sources with the new method, obtaining reliable flaring configurations. Finally, we point out that other stopping rules traditionally applied in the case of Poisson data do not properly work. A systematically validation of this approach is under construction in both astronomical and medical imaging applications, in the case the noise is Poisson, the algorithm adopted is EM and the model equations is genuinely ill-posed.

\section*{Acknowledgments}

The work has been supported by the EU FP7 HESPE grant no. 263086.

\section*{References}

\bibliography{EM_STOPPING_RULE.bib}

\end{document}